\documentclass[11pt,a4paper]{article}

\usepackage{amsmath,amssymb,amscd,amsthm}
\usepackage{graphics}
\input xy
\xyoption{all}

\usepackage{comment}
\usepackage{tikz}

\usetikzlibrary{arrows}

\usepackage[paperwidth=210mm,
            paperheight=297mm,
            a4paper,
            twoside,
            top=1in, left=1.2in, right=1.2in, bottom=1in]{geometry}

\usepackage[active]{srcltx}

\newtheorem{lemma}{Lemma}

\newtheorem{prop}[lemma]{Proposition}
\newtheorem{cor}[lemma]{Corollary}

\newtheorem{defi}[lemma]{Definition}
\newtheorem{thm}[lemma]{Theorem}

\newenvironment{example}[1][Example]{\begin{trivlist}
\item[\hskip \labelsep {\bfseries #1}]}{\end{trivlist}}

\usepackage{float}
\restylefloat{figure}

\begin{document}

\title{Enumeration of basic ideals in type $B$}
\author{Jonathan Nilsson}
\date{\today}
\maketitle

\begin{abstract}
\noindent The number of ad-nilpotent ideals of the Borel subalgebra of the classical Lie algebra of type $B_{n}$ is determined using combinatorial arguments involving a generalization of Dyck-paths.
 We also solve a similar problem for the untwisted affine Lie algebra of type $\tilde{B}_{n}$,
 where we instead enumerate a certain class of ideals called \emph{basic ideals}. This leads to an explicit formula for the number of basic ideals in $\tilde{B}_{n}$, which gives rise to a new integer sequence.
\end{abstract}

\section{Introduction}
Let $\mathfrak{g}$ be a simple complex Lie algebra with a
 fixed triangular decomposition
 \[\mathfrak{g} = \mathfrak{n}_{-} \oplus \mathfrak{h} \oplus \mathfrak{n}_{+}\] in the sense of Moody and Pianzola~\cite{MP} and let
 $\mathfrak{b}=\mathfrak{h} \oplus \mathfrak{n}_{+}$ be the corresponding Borel subalgebra.
 An ideal $\mathfrak{i}$ of $\mathfrak{b}$ is called \emph{ad-nilpotent} provided that the adjoint action of each element of $\mathfrak{i}$ on $\mathfrak{b}$ is nilpotent.
 When $\mathfrak{g}=\mathfrak{sl}_{n}(\mathbb{C})$, the set of strictly upper triangular matrices is a typical
 example of an ad-nilpotent ideal, but there are many other such ideals contained in this one.
 This leads to the problem of enumerating ad-nilpotent ideals, which has been done by Krattenthaler, Orsina and Papi~\cite{KOP} for all simple finite dimensional Lie algebras.
 See also \cite{CP1,CP2,AKOP}.
 For example, when $\mathfrak{g}=\mathfrak{sl}_{n}(\mathbb{C})$, the number of ad-nilpotent ideals in $\mathfrak{b}$ is given by the $n$'th Catalan number $\frac{1}{n+1} \binom{2n}{n}$.

In the case when $\mathfrak{g}$ is an untwisted affine Lie algebra, any ad-nilpotent ideal of $\mathfrak{b}$ is contained in the center of $\mathfrak{g}$
 so the corresponding enumeration problem becomes trivial.
 Baur and Mazorchuk~\cite{BaMa} formulated another problem, where instead of ad-nilpotent ideals, one considers the so called \emph{basic ideals} (see the definition in Section~\ref{s4}).
The enumeration problem for basic ideals was solved for the affine algebra $\tilde{\mathfrak{sl}}_{n}$ which led to a new integer sequence.

In the present paper we solve the enumeration problem for basic ideals for untwisted affine Lie algebras of affine type $\tilde{B}_{n}$. We prove the following statement.

\begin{thm}
\label{thmmain}
The number $\tilde{b}_{n}$ of basic ideals in the Lie algebra of type $\tilde{B}_{n}$ is given by the formula
\[\tilde{b}_{n} = (3n+5)2^{2n-2} - 2(3n-1)\binom{2n-2}{n-1}.\]
\end{thm}

To prove Theorem~\ref{thmmain} we establish a combinatorial scheme in which basic ideals for type $B$ are encoded using certain pairs of type $B$-analogues for Dyck-paths.
This gives an explicit but complicated formula for the number of basic ideals. The simplification of this expression to the simple form in Theorem~\ref{thmmain} is a nontrivial
 combinatorial computation which was performed by Christian Krattenthaler. The above simple formula, in particular, implies that the sequence $\tilde{b}_{n}$
 satisfies the following non-homogeneous linear recurrence relation. \[\tilde{b}_{n}-8\tilde{b}_{n-1}+16\tilde{b}_{n-2} = \frac{24}{n-1} \binom{2n-6}{n-2}, \quad n \geq 4.\] We do not have any direct combinatorial explanation
 for this recursion.

The paper is organized as follows. In Section~\ref{s2} we focus on Lie algebras of regular type $B$. We introduce \emph{$B$-paths}, an analogue for Dyck paths in type $B$, and show that the set of
 such paths bijectively corresponds to the set of ad-nilpotent ideals in type $B$. We use this to show that the number of ad-nilpotent ideals in the Lie algebra of type $B_{n}$ equals $\binom{2n}{n}$, the
 type $B$ Catalan number, 
 which agrees with a previous result~\cite{KOP}. After this we want to generalize these ideas to the affine case. 
 The untwisted affine Lie algebra of type $B$ is introduced in Section~\ref{s3} and in Section~\ref{s4} we study its root system structure. Following Baur and Mazorchuk~\cite{BaMa}, we then proceed to define the concept of
 a basic ideal, our affine analogue of a nilpotent ideal, and in Section~\ref{s5} we show that there is a bijective correspondence between basic ideals and a
 certain set of pairs of $B$-paths called \emph{admissible} pairs. This reduces our problem to
 an enumeration of the admissible pairs. Finally, in Section~\ref{s6} we write down an explicit formula for the number of admissible pairs by considering a number of cases.
 We then proceed to simplify this expression using a number of combinatorial lemmas, and after a lot of simplification we arrive at the short expression of Theorem \ref{thmmain}.

\section{The simple algebra of type $B$}
\label{s2}
\subsection{Root system structure}
Following the notation of Bourbaki~\cite{Bo}, the root system of type $B_{n}$ can be constructed in $\mathbb{R}^{n}$ as follows. 
Let $\{e_{i}\}_{i=1}^{n}$ be the standard orthogonal basis for $\mathbb{R}^{n}$. For $1 \leq i \leq n-1$ define $\alpha_{i}=e_{i}-e_{i+1}$, and let $\alpha_{n}=e_{n}$.
Then the (simple roots) $\alpha_{1}, \ldots , \alpha_{n}$ constitutes a base for a root system of type $B_{n}$ where the positive roots are given by
\[\sum_{j \leq i \leq k} \alpha_{i},  \quad  1 \leq j \leq k \leq n;   \qquad \sum_{j \leq i \leq k} \alpha_{i} +  \sum_{k+1 \leq i \leq n} 2\alpha_{i}, \quad 1 \leq j < k < n.\]
We identify each such root with the integer pair $(r,s)$, such that $r$ is the minimal value of $i$ for which the $\alpha_{i}$-coefficient of the
 root in the simple basis is non-zero, and $s=r+h-1$ where $h$ is the height of the root (the sum of its coefficients in the simple basis).
 For example, in $B_{7}$ we write $(3,10)$ for the root $(0,0,1,1,2,2,2)$, and $(2,5)$ for $(0,1,1,1,1,0,0)$.

Positive roots are partially ordered via
\[\alpha \prec \beta \Longleftrightarrow \; \beta - \alpha \text{ is a } \mathbb{Z}_{\geq 0} \text{-linear combination of simple roots.}\]
This order can of course be extended to all of $\mathfrak{h}^{*}$. It is easy to check that $(r,s) \prec (\hat{r},\hat{s})$ if and only if $\hat{r} \leq r$ and
 $\hat{s} \geq s$.
The ($\prec$)-poset structure on positive roots can now be visualized in a modified version of the Hasse-diagram. Consider a diagram consisting of rows and columns.
 If $(i,j)$ is a root then we write it in the
 $i$'th row (from top to bottom) and $j$'th column (from left to right), in other words, we write $(i,j)$ in the $(i,j)$'th position. 
Then $\prec$-greater roots are always in the north and east directions.

Not all positions in the diagram will contain roots. We obviously have $i \geq 1$ and  $j \geq 1$. We must also have $i \leq j$, since positive roots have positive height,
 and we must have $i \leq 2n-j$ (compare the original expression of a general positive root). Thus, all roots lie in a triangular scheme. For example, in $B_{4}$ the positive roots are ordered as follows.
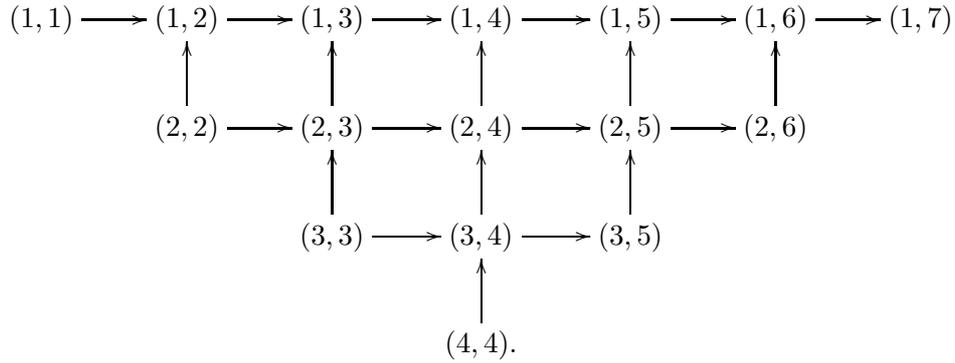
\vspace{-10mm}\[ \xymatrix{                                                                                                                                                \\
(1,1) \ar[r]     & (1,2)   \ar[r]       & (1,3) \ar[r]         & (1,4)  \ar[r]        & (1,5) \ar[r]      & (1,6) \ar[r]         & (1,7)     \\
                 & (2,2) \ar[r] \ar[u]   & (2,3) \ar[r] \ar[u]  & (2,4) \ar[r] \ar[u]  & (2,5) \ar[r] \ar[u]  & (2,6)  \ar[u]\\
                 &                      & (3,3) \ar[r] \ar[u] & (3,4) \ar[r] \ar[u] & (3,5) \ar[u]  \\
                 &                      &                      & (4,4).  \ar[u]            }
\]
Here $\alpha \rightarrow \beta$ means that $\beta$ covers $\alpha$ with respect to $\prec$, that is $\beta$ is a minimal element of the set $\{\gamma | \alpha \prec \gamma\}$.
Thus we have $\alpha \prec \beta$ if and only if there is
 a directed path from $\alpha$ to $\beta$.

Any ad-nilpotent ideal $\mathfrak{i}$ in $\mathfrak{n}_{+}$ can be decomposed as a direct sum of some root spaces corresponding to positive roots.
 If $\mathfrak{i}$ contains some nonzero $x$ from a root space $L_{\gamma}$, it contains all of
 $L_{\gamma}$ since the root spaces are one-dimensional. Successively commuting $x$ with elements from different root spaces $L_{\alpha_{i}}$ we can obtain elements from
 any root space $L_{\gamma'}$ where $\gamma' \succ \gamma$, and hence $\mathfrak{i}$ also contains all of these root spaces.
 This shows that an ad-nilpotent ideal in the Borel subalgebra of $B_{n}$ corresponds precisely to a coideal of the poset structure above (recall that a coideal of a poset is subposet $S$
 such that $\alpha \in S$ and $\beta \geq \alpha$ implies $\beta \in S$). Since all the arrows above point either up or right,
 such a subset can be specified by choosing a path in the diagram, going in the south and east directions partitioning the roots.
 Counting the ad-nilpotent ideals then corresponds to counting the number of such paths.

\subsection{$B$-paths}
To formalize this idea of a path, for the algebra of type $B_{n}$, consider a $(2n-1) \times n$ rectangle in $\mathbb{R}^{2}$
 with the north-west corner in $(0,0)$ and the south-east corner in $(2n-1,-n)$. This rectangle contains the triangular shape above in the
 natural way if we think of the roots as $1\times 1$-boxes in
 the $\mathbb{Z} \times \mathbb{Z}$-lattice. For example, in the case $B_{4}$ we have

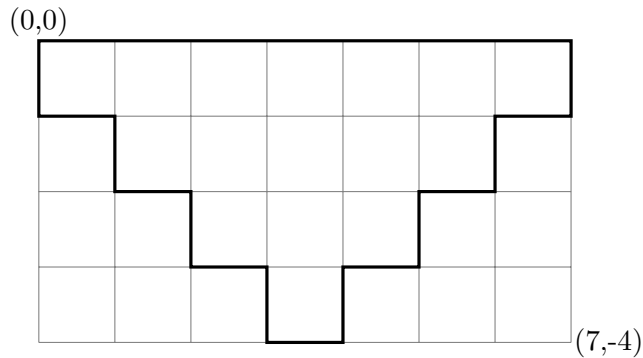
\[
\begin{tikzpicture}
  \draw[color=gray] (0,-4) grid (7,0);
  \node (a) at (0,0.25) {(0,0)};
  \node (a) at (7.5,-4) {(7,-4)};
  \draw[very thick] (0,0) -- (1,0) -- (2,0) -- (3,0) -- (4,0) -- (5,0) -- (6,0) -- (7,0) 
      -- (7,-1) -- (6,-1) -- (6,-2) -- (5,-2) -- (5,-3) -- (4,-3) -- (4,-4) -- (3,-4) 
       -- (3,-3) -- (2,-3) -- (2,-2) -- (1,-2) -- (1,-1) -- (0,-1) --cycle;
\end{tikzpicture}
\]

\begin{defi} A $B_{n}$-path or just a  $B$-path if $n$ is understood, is
 a word of length $2n$ on the alphabet $\{r,f\}$ such that each prefix of this word contains at least as many $r$'s as $f$'s.
\end{defi}

Note that this generalizes what is known as a {\bf Dyck-path}. A Dyck-path is a $B$-path with an equal number of $f$'s and $r$'s, see Grimaldi~\cite{GRI}.

\begin{example}
$rrffrrfrrr$ is a $B_{5}$-path and $rrrrrrrr$ is a $B_{4}$-path, but $rrfrffrffrrrrr$ is not a $B_{7}$-path since the first $9$ letters contain more $f$'s than $r$'s.  
\end{example}

We shall sometimes use the notation $rrrrr=r^{5}$ and so on.

Let $\mathfrak{g}_{n}$ be a Lie algebra of type $B_{n}$ and let $\mathfrak{b}_{n}$ be a Borel subalgebra in $\mathfrak{g}_{n}$.

\begin{prop}
\label{pathprop}
There is a 1-1 correspondence between the set of all ad-nilpotent ideals of $\mathfrak{b}_{n}$ and the set of all $B_{n}$-paths.
\end{prop}

\begin{proof}
Let $\mathfrak{i}$ be an ad-nilpotent ideal in $\mathfrak{b}_{n}$. Consider the $(2n-1) \times n$ rectangular picture described above
 in which boxes are identified with positive roots of $\mathfrak{g}_{n}$. If $\mathfrak{i}$ contains a root element for
 some root $\alpha$, the fact that $\mathfrak{i}$ is an ideal implies that $\mathfrak{i}$ contains root elements for all roots $\beta$ such that
 $\beta \succ \alpha$. From this it follows that there is a unique path $p=p(\mathfrak{i})$ which has the following properties.
 \begin{itemize}
  \item  $p$ starts at $(-1,0)$;
  \item each step of $p$ goes to the right (along $(1,0)$) or down (along $(0,-1)$);
  \item $p$ terminates at a point of the form $(2n-1-x,-x)$, $x \in \{0,1, \ldots , n\}$;
  \item $p$ separates all boxes corresponding to elements of $\mathfrak{i}$ from all other boxes.
 \end{itemize}
Substituting each $(1,0)$-step in $p$ with $r$ and each $(0,-1)$-step in $p$ with $f$, we obtain a $B_{n}$-path.

Conversely, given a $B_{n}$-path we make the reverse substitution and obtain a path satisfying the first three conditions above.
 It is easy to see that the linear span of root elements corresponding to the roots lying to the northeast of this path is an ad-nilpotent ideal
 of $\mathfrak{b}_{n}$. The claim follows.
\end{proof}

Here is an explicit example in type $B_{4}$. The $B_{4}$-path $rrrfrfrr$ is drawn below, it corresponds to the ideal spanned by the root spaces of the roots
 $(1,3)$, $(1,4)$, $(1,5)$, $(1,6)$, $(1,7)$, $(2,4)$, $(2,5)$ and $(2,6)$.
 Also note that any pair of (non-zero) elements from the root spaces of the roots $(1,3)$ and $(2,4)$ generates this ideal.
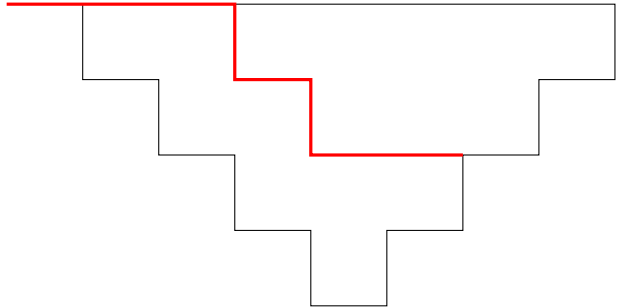
\[
\begin{tikzpicture}
  \draw (0,0) -- (1,0) -- (2,0) -- (3,0) -- (4,0) -- (5,0) -- (6,0) -- (7,0) 
      -- (7,-1) -- (6,-1) -- (6,-2) -- (5,-2) -- (5,-3) -- (4,-3) -- (4,-4) -- (3,-4) 
       -- (3,-3) -- (2,-3) -- (2,-2) -- (1,-2) -- (1,-1) -- (0,-1) --cycle;
  \draw[very thick,color=red] (-1,0) -- (0,0) -- (1,0) -- (2,0) -- (2,-1) -- (3,-1) -- (3,-2) -- (4,-2) -- (5,-2);
\end{tikzpicture}
\] 
The path $rfrfrfrf$ gives the improper coideal containing all the roots, and the path $rrrrrrrr$ corresponds to the empty coideal containing no roots.

We can now state a corollary regarding the number of ad-nilpotent ideals in $B_{n}$. This result was proved in another way by Krattenthaler, Orsina, and Papi~\cite{KOP}.

\begin{cor}
  The number of $B_{n}$-paths, hence also the number of ad-nilpotent ideals in $\mathfrak{b}_{n}$, equals $\binom{2n}{n}$.
\end{cor}

\begin{proof} A $B_{n}$-path can be extended to a $B_{n+1}$-path by appending two letters on the right, and it is clear that any $B_{n+1}$-path can be obtained this way.
There are four possible choices of what to append: $rr$, $rf$, $fr$, $ff$ and each of these gives a $B_{n+1}$-path, \emph{except} when the original word
 consisted of as many $f$'s as $r$'s - in this case $fr$ and $ff$ violate the condition of being a $B$-path (compare with the figure below where
 the top path $rrrrrfrr$ may be completed in $4$ ways, but the bottom path $rfrrrfff$ in only $2$ ways,
 when $n$ increases from $4$ to $5$)
\[
\begin{tikzpicture}
  \draw (0,0) -- (1,0) -- (2,0) -- (3,0) -- (4,0) -- (5,0) -- (6,0) -- (7,0) -- (8,0) -- (9,0) 
 -- (9,-1) -- (8,-1) -- (8,-2) -- (7,-2) -- (7,-3) -- (6,-3) -- (6,-4) -- (5,-4) -- (5,-5) -- (4,-5) -- (4,-4)
 -- (3,-4)  -- (3,-3) -- (2,-3) -- (2,-2) -- (1,-2) -- (1,-1) -- (0,-1) --cycle;
  \draw  (7,0) -- (7,-1) -- (6,-1) -- (6,-2) -- (5,-2) -- (5,-3) -- (4,-3) -- (4,-4); 
  \draw (7,0) -- (8,0) -- (9,0); 
  \draw[very thick,color=red]  (-1,0) -- (0,0) -- (1,0) -- (2,0) -- (3,0) -- (4,0) -- (4,-1) -- (5,-1) -- (6,-1);
  \draw[very thick,color=blue] (-1,0) -- (0,0) -- (0,-1) -- (1,-1) -- (2,-1) -- (3,-1) -- (3,-2) -- (3,-3) -- (3,-4);
  \draw[very thick,color=blue, dashed] (3,-4)--(5,-4);
  \draw[very thick,color=blue, dashed] (4,-4)--(4,-5);
  \draw[very thick,color=red, dashed] (8,-1)--(6,-1)--(6,-3);
  \draw[very thick,color=red, dashed] (7,-1)--(7,-2)--(6,-2);
\end{tikzpicture}
\]
 A $B_{n}$-path where there are an equal number
 of $f$'s and $r$'s is precisely a Dyck-path of semilength $n$, and there are $\frac{1}{n+1}\binom{2n}{n}$ of these (see for example Grimaldi~\cite{GRI}). 

Assuming we have $\binom{2n}{n}$ $B_{n}$-paths, the number of $B_{n+1}$-paths is thus given by
 \[4\binom{2n}{n} - \frac{2}{n+1}\binom{2n}{n} = \binom{2(n+1)}{n+1}.\]
 Hence, by induction, the proof is complete once we note that
 it is true for $B_{2}$, and it is easy to check that we have precisely $\binom{4}{2}=6$ paths then.
 This completes the proof.
\end{proof}

\section{The affine algebra of type $B$}
\label{s3}
The situation above has an analogue when we instead of the type $B_{n}$ consider the untwisted affine Lie algebra
 $\tilde{\mathfrak{g}}_{n}$ of type $\tilde{B}_{n}$.
The algebra $\tilde{\mathfrak{g}}_{n}$ is obtained as follows. First consider the loop algebra
 $\mathfrak{g}_{n} \otimes \mathbb{C}[t,t^{-1}]$ endowed with the bracket $[x\otimes t^{m},y\otimes t^{n}] = [x,y] \otimes t^{m+n}$.
The universal central extension of the Loop algebra is explicitly, $(\mathfrak{g}_{n} \otimes \mathbb{C}[t,t^{-1}]) \oplus \mathbb{C}c$
 where the bracket is modified as follows:
\[[x \otimes t^{m} + \lambda c, y \otimes t^{n} + \mu c] = [x,y] \otimes t^{m+n} +  m (x|y) \delta_{m+n,0} c\] where $(\cdot | \cdot )$ is the Killing form of $\mathfrak{g}_{n}$.
The algebra $\tilde{\mathfrak{g}}_{n}$ is now obtained by extending the latter by a derivation $d$ which acts as follows
 $[d,(x\otimes t^{m})] = m (x \otimes t^{m})$ and $[d,c]=0$. 
The resulting Lie algebra $\mathfrak{\tilde{g}}_{n}$ is called the {\bf untwisted affine Kac-Moody algebra} corresponding to $\mathfrak{g}_{n}$. 
The algebra $\tilde{\mathfrak{g}}_{n}$ has a triangular decomposition
 $\mathfrak{\tilde{g}_{n}} = \mathfrak{\tilde{n}}_{-} \oplus \mathfrak{\tilde{h}} \oplus \mathfrak{\tilde{n}}_{+}$, where
 \[\mathfrak{\tilde{n}}_{+} = Span(\{x \otimes 1| x \in \mathfrak{n_{+}}\} \cup \{x \otimes t^{k} | x \in \mathfrak{g}_{n}, k>0\})\]
\[\mathfrak{\tilde{n}}_{-} = Span(\{y \otimes 1| y \in \mathfrak{n_{-}}\} \cup \{y \otimes t^{k} | y \in \mathfrak{g}_{n}, k>0\})\]
 \[\mathfrak{\tilde{h}} = Span(\{c,d\} \cup \{h \otimes 1 | h \in \mathfrak{h}\}).\]

The corresponding Borel subalgebra $\mathfrak{\tilde{b}_{n}} = \mathfrak{\tilde{h}} \oplus \mathfrak{\tilde{n}}_{+}$ no longer contains any nontrivial ad-nilpotent ideals,
 so instead we consider {\bf combinatorial ideals}, that is, ideals $\mathfrak{i}$ of $\mathfrak{\tilde{b}_{n}}$ contained in $\mathfrak{\tilde{n}}_{+}$
 which have finite codimension and are unions of entire root spaces:
 $\mathfrak{i} \cap \tilde{\mathfrak{g}}_{\alpha} \in \{ \{0\}, \tilde{\mathfrak{g}}_{\alpha}\}$ for each positive root $\alpha$.
For any combinatorial ideal $\mathfrak{i}$ we define $supp(\mathfrak{i})$, called the \emph{support} of $\mathfrak{i}$, to
 be the set of positive roots whose root spaces are are contained in $\mathfrak{i}$, that is, $supp(\mathfrak{i})=\{ \alpha \in \mathfrak{\tilde{h}}^{*} | \tilde{\mathfrak{g}}_{\alpha} \subset \mathfrak{i}\}$.

\section{Root system structure in type $\tilde{B}_{n}$}
\label{s4}
To increase readability we shall sometimes identify a root system with its type.
Recall that ad-nilpotent ideals in the Borel algebra for type $B_{n}$ were uniquely determined by the coideals of the poset of positive roots.
Both $B_{n}$ and $\tilde{B}_{n}$ are partially ordered with respect to the order $\prec$ defined with respect to the corresponding sets of simple roots.

Let $\delta \in \tilde{B}_{n}$ be the indivisible positive imaginary root. Then any imaginary root of $\tilde{B}_{n}$ can be written as $k \delta$,
 for $k \in \mathbb{Z} \setminus \{0\}$ and
 any real root of $\tilde{B}_{n}$ can be written as $\gamma + k \delta$ for some $\gamma \in B_{n}$ and some $k \in \mathbb{Z}$.
Let $\alpha_{max}$ be the maximal root in $B_{n}$ and defining $\alpha_{0} = \delta -\alpha_{max}$. Then $\alpha_{0}, \alpha_{1}, \ldots, \alpha_{n}$ is
 a base for $\tilde{B}_{n}$. Let $D_{1}=\{\delta\}$, let $D_{2}$ be the set of positive roots of $B_{n}$, and let
 $D_{3}$ be the set of roots of form $\delta + \gamma$ where $\gamma$ is a negative root of $B_{n}$. These three sets are disjoint, and we define $D:=D_{1} \cup D_{2} \cup D_{3}$.
 Then the positive roots of $\tilde{B}_{n}$ are the roots of form
 $d + k \delta$ for $d \in D$, $k \geq 0$, and the negative roots correspond to $k<0$ (except $0$ which is not a root at all).

A combinatorial ideal $\mathfrak{i}$ whose support (i.e. the set of roots $\alpha$ for which
 $\tilde{\mathfrak{g}}_{\alpha} \subset \mathfrak{i}$ ) intersects $D$ nontrivially is called a {\bf basic ideal}.
 Baur and Mazorchuk~\cite{BaMa} showed that any combinatorial
 ideal can be obtained from a unique basic ideal through ``translation'' in the $\delta$-direction. The number of basic ideals
 is what we are going to count in the remainder of this paper.

We now describe the poset structure of $D$ in terms of its partition above. We have $D_{1}=\{\delta\}$ and $\delta$ is the maximum element of the poset $D$.
$D_{2}$ is the set of positive roots of $B_{n}$ which is fully embedded into $D$ as a poset. Thus the poset structure on $D_{2}$ is precisely as in Section~\ref{s2}.
$D_{3}$ consists of roots of form $\delta - \alpha$ where $\alpha$ is a positive root of $B_{n}$.
 These are ordered in the reversed way,
 since $\alpha \prec \beta \Leftrightarrow \delta - \alpha \succ \delta - \beta$. The roots of $D_{3}$ can be organized into a digram similarly
 to how we organized positive roots of $B_{n}$ but reflected with respect to the horizontal line. More explicitly,
 in the position $(i,j)$ we write $\delta - (n-i+1,2n-j)$ whenever $(n-i+1,2n-j)$ is a positive root. Then for $\hat{i} \leq i$ and $\hat{j} \geq j$,
 $\delta - (n-\hat{i}+1,2n-\hat{j}) \succ \delta - (n-i+1,2n-j)$, so greater elements are again in the north and east directions.
 Putting the diagrams $D_{2}$ and $D_{3}$ together
 we obtain a diagram which is illustrated below in the case $n=4$.
\[
    \xymatrix@C=4mm@R=5mm{ 
                 &                     &                     & \delta-(4,4)               &                      &                      &  \\
                 &                      & \delta-(3,5) \ar[r]        & \delta-(3,4) \ar[r] \ar[u] & \delta-(3,3)      \\
                 & \delta-(2,6) \ar[r]        & \delta-(2,5) \ar[r] \ar[u] & \delta-(2,4) \ar[r] \ar[u] & \delta-(2,3) \ar[r] \ar[u] & \delta-(2,2)  \\
\delta-(1,7) \ar[r]    & \delta-(1,6) \ar[r] \ar[u] & \delta-(1,5) \ar[r] \ar[u] & \delta-(1,4) \ar[r] \ar[u] & \delta-(1,3) \ar[r] \ar[u] & \delta-(1,2) \ar[r] \ar[u] & \delta-(1,1)     \\
                                                                                                                                                  \\
(1,1) \ar[r]     & (1,2)   \ar[r]       & (1,3) \ar[r]         & (1,4)  \ar[r]        & (1,5) \ar[r]      & (1,6) \ar[r]         & (1,7)     \\
                 & (2,2) \ar[r] \ar[u]   & (2,3) \ar[r] \ar[u]  & (2,4) \ar[r] \ar[u]  & (2,5) \ar[r] \ar[u]  & (2,6)  \ar[u]\\
                 &                      & (3,3) \ar[r] \ar[u] & (3,4) \ar[r] \ar[u] & (3,5) \ar[u]  \\
                 &                      &                      & (4,4)  \ar[u]            }
\]

 However, we have yet to consider $\prec$-relations between $D_{2}$ and $D_{3}$ in the diagram.
 Because of the $\alpha_{0}$-coordinate, a
 root in the upper part can not be $\prec$-smaller than a root in the lower part. The following lemma specifies precisely when a root in the lower part is $\prec$-smaller than a root in the upper part.

\begin{lemma}
For $i>1$ we have $(i,j) \prec \delta - (\hat{i},\hat{j})$ if and only if $\hat{j} \leq 2n-j$, while
$(1,j) \prec \delta - (\hat{i},\hat{j})$ if and only if $\hat{j} \leq 2n-j \text{ and } \hat{i} > 1$.
\end{lemma}

\begin{proof} We have $(i,j) \prec \delta - (\hat{i},\hat{j}) \Leftrightarrow (i,j)+(\hat{i},\hat{j}) \prec \alpha_{0} + \alpha_{max}
 \Leftrightarrow (i,j)+(\hat{i},\hat{j}) \prec \alpha_{max}$ (note that $(i,j)+(\hat{i},\hat{j})$ does not have to be a root).
 Thus, for a given root $(i,j)$, we need only consider which roots $(\hat{i},\hat{j})$
 can be added such that the sum is still $\prec$-smaller than $\alpha_{max} = \alpha_{1} + 2\alpha_{2} + 2\alpha_{3} + \cdots + 2 \alpha_{n} = (1,2,2, \ldots ,2,2)$.
 First, if $j \geq n$, both $(i,j)$ and $(\hat{i},\hat{j})$ have the form $(0,0, \ldots ,0,0,1,1, \ldots ,1,1,2,2, \ldots ,2,2)$ in the basis of simple roots.
 Moreover, for $(i,j)$ the number of $2$'s at the end equals $j-n$.  The sum $(i,j) + (\hat{i},\hat{j})$ is not allowed to have
 any coordinates equal to $3$, which implies $\hat{j} \leq n-(j-n) \Leftrightarrow \hat{j} \leq 2n-j$.
 If $j \leq n$, the root $(i,j)$ has the form $(0,0, \ldots ,0,0,1,1, \ldots ,1,1,0,0, \ldots ,0,0)$
 where the number of
 $0$'s in the end equals $n-j$, so similar arguments imply that $(\hat{i},\hat{j})$ can have at most $n-j$ coordinates equal to $2$ at the end, or equivalently
 $\hat{j} \leq n + (n-j) \Leftrightarrow \hat{j} \leq 2n-j$.

Thus a necessary condition for $(i,j) \prec \delta - (\hat{i},\hat{j})$ is that $\hat{j} \leq 2n-j$. This is also sufficient \emph{except} for the
 special case when $i=1$. In that case, since the $\alpha_{1}$-coefficient in $\alpha_{max}$ and $(i,j)$ is $1$, we can not have $\hat{i}=1$.
 Taken together, this proves the claim.
\end{proof}

Let us translate this result to our diagram above. Since the $(i,j)$'th position in the upper part contains the root $\delta - (n-i+1,2n-j)$, any element in
 the $j$'th column (except the one in the first row) at the lower part is $\prec$-smaller than all elements in the $j$'th column in the upper part.
 The root $(1,j)$ in the first row and $j$'th column is $\prec$-smaller than all elements in the $j$'th column above, except for the one in the bottom row.
 For example, our $\tilde{B_{4}}$ diagram is now completed as follows:
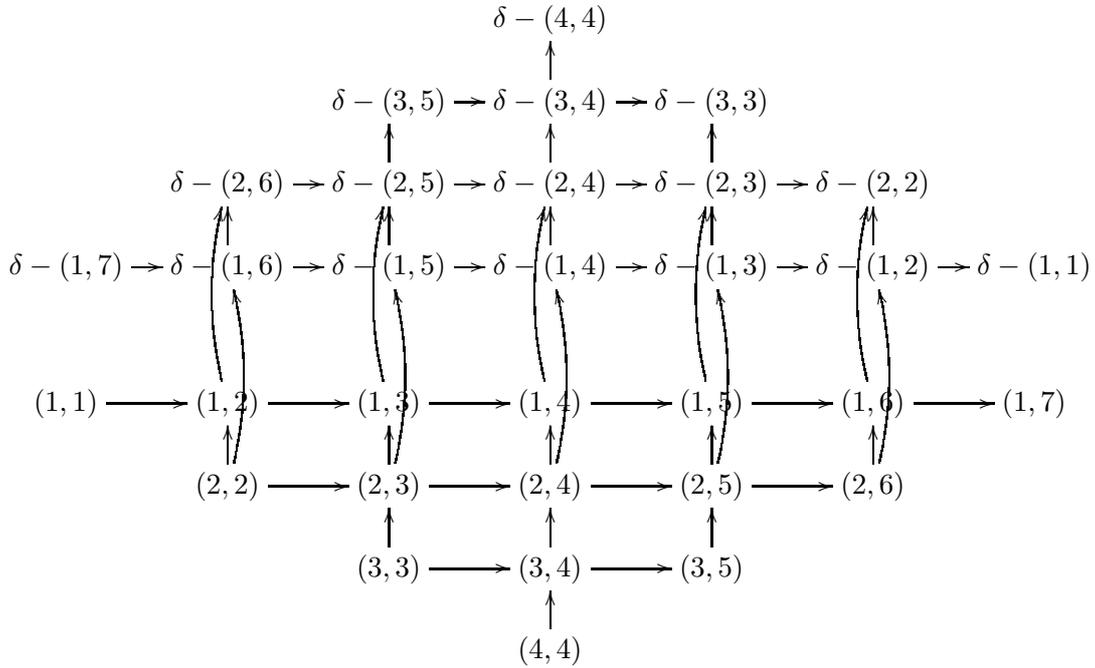
\[
    \xymatrix@C=4mm@R=5mm{ 
                 &                     &                     & \delta-(4,4)               &                      &                      &  \\
                 &                      & \delta-(3,5) \ar[r]        & \delta-(3,4) \ar[r] \ar[u] & \delta-(3,3)      \\
                 & \delta-(2,6) \ar[r]        & \delta-(2,5) \ar[r] \ar[u] & \delta-(2,4) \ar[r] \ar[u] & \delta-(2,3) \ar[r] \ar[u] & \delta-(2,2)  \\
\delta-(1,7) \ar[r]    & \delta-(1,6) \ar[r] \ar[u] & \delta-(1,5) \ar[r] \ar[u] & \delta-(1,4) \ar[r] \ar[u] & \delta-(1,3) \ar[r] \ar[u] & \delta-(1,2) \ar[r] \ar[u] & \delta-(1,1)     \\
                                                                                                                                                  \\
(1,1) \ar[r]     & (1,2)   \ar[r]  \ar@/^/[uuu]      & (1,3) \ar[r]   \ar@/^/[uuu]      & (1,4)  \ar[r]  \ar@/^/[uuu]      & (1,5) \ar[r] \ar@/^/[uuu]     & (1,6) \ar[r]  \ar@/^/[uuu]       & (1,7)     \\
                 & (2,2) \ar[r] \ar[u]  \ar@/_/[uuu]  & (2,3) \ar[r] \ar[u]  \ar@/_/[uuu] & (2,4) \ar[r] \ar[u]  \ar@/_/[uuu]  & (2,5) \ar[r] \ar[u]  \ar@/_/[uuu] & (2,6)  \ar[u]  \ar@/_/[uuu] \\
                 &                      & (3,3) \ar[r] \ar[u] & (3,4) \ar[r] \ar[u] & (3,5) \ar[u]  \\
                 &                      &                      & (4,4)  \ar[u]            }
\]
Here again, a root is $\prec$-smaller than another root if and only if there is a path of arrows from the first one to the second one.
Note that the top row of the bottom part is linked to the second to bottom row in the upper part and vice versa.

Our goal is now the express the number of coideals in such a poset (for an arbitrary $n$) in terms of pairs of paths. To this end, construct again
 a rectangle in $\mathbb{R}^{2}$,
 this time with the northwest corner in $(0,n+1)$ and the southeast corner in $(2n-1,-n)$, where $1 \times 1$-blocks correspond to roots of $D \setminus \{\delta\}$. The rectangle contains the two triangles in a natural way
 as illustrated below in the case $n=4$.
\[
\begin{tikzpicture}
  \draw (0,1) -- (1,1) -- (2,1) -- (3,1) -- (4,1) -- (5,1) -- (6,1) -- (7,1) 
      -- (7,2) -- (6,2) -- (6,3) -- (5,3) -- (5,4) -- (4,4) -- (4,5) -- (3,5) 
       -- (3,4) -- (2,4) -- (2,3) -- (1,3) -- (1,2) -- (0,2) --cycle;

  \draw (0,0) -- (1,0) -- (2,0) -- (3,0) -- (4,0) -- (5,0) -- (6,0) -- (7,0) 
      -- (7,-1) -- (6,-1) -- (6,-2) -- (5,-2) -- (5,-3) -- (4,-3) -- (4,-4) -- (3,-4) 
       -- (3,-3) -- (2,-3) -- (2,-2) -- (1,-2) -- (1,-1) -- (0,-1) --cycle;
  \draw[very thick,color=red] (-1,0) -- (0,0) -- (1,0) -- (2,0) -- (2,-1) -- (3,-1) -- (3,-2) -- (4,-2) -- (5,-2);
  \draw[very thick,color=blue] (8,1) -- (7,1) -- (6,1) -- (5,1) -- (4,1) -- (3,1)  -- (2,1)  -- (2,2)  -- (2,3);
  \node (a) at (2.5,-0.5) {(1)};
  \node (b) at (2.5,2.5) {(2)};
    \draw[dashed] (2,0) -- (3,0) -- (3,-1) -- (2,-1) -- cycle;
    \draw[dashed] (2,2) -- (3,2) -- (3,3) -- (2,3) -- cycle;

\end{tikzpicture}
\]
We have drawn two paths in the picture,  a red one and a blue one. These paths together specify a coideal of the poset in the following way:
 we take all roots corresponding to boxes to the northeast of the blue path together with all roots corresponding to boxes from the lower triangle
 to the northeast of the red path. 
 Every coideal can be obtained this way, but the converse is not true, some choices of paths do not correspond to coideals. For example, in the above diagram, since the red path specifies that the root in box marked by $(1)$ belongs to the coideal, the root in the box
 marked by $(2)$ in the top triangle should be in the coideal as well.
 Thus the blue path can not be taken  ``too high'' given the red path. The choice in the picture above indeed does corresponds to a coideal.
 A pair of paths which do correspond to a coideal will be called an \emph{admissible} pair of paths.

\section{Admissible pairs of paths}
\label{s5}
We now formalize the correspondence between basic ideals in $\mathfrak{b}_{n}$ and admissible pairs
 of $B_{n}$-paths by generalizing
 what was done in Proposition~\ref{pathprop}.

Let $\mathfrak{i}$ be a basic ideal in $\mathfrak{b}_{n}$.
 Let $D_{\mathfrak{i}} :=D \cap supp(\mathfrak{i})$ where $D = D_{1} \cup D_{2} \cup D_{3}$ as in Section~\ref{s4}.
 Then $D_{\mathfrak{i}} = (D_{1}\cap supp(\mathfrak{i})) \cup (D_{2}\cap supp(\mathfrak{i})) \cup (D_{3} \cap supp(\mathfrak{i}))$.
 Since $\mathfrak{i}$ is assumed to be basic, and $\delta$ is the maximum element in the poset $D$, we have
 $D_{1}\cap supp(\mathfrak{i}) = \{\delta\} \cap supp(\mathfrak{i}) = \{\delta\}$. So the sets $D_{2}\cap supp(\mathfrak{i})$ and $D_{3} \cap supp(\mathfrak{i})$ determine the basic ideal.
Now, consider again the geometric scheme discussed above in which we have a rectangle in $\mathbb{R}^{2}$ with northwest corner in $(0,n+1)$ and the southeast corner in $(2n-1,-n)$
 which contains $1\times 1$-boxes corresponding to elements of $D \setminus \{\delta\}$.
 The basic ideal $\mathfrak{i}$ corresponds to a subset of such boxes (the empty subset in case $supp(\mathfrak{i}) \cap D = \{ \delta \}$).
 By the same arguments as in Proposition~\ref{pathprop} it follows that there exists a unique pair of paths $(p,q)$ which has the following properties.
 \begin{itemize}
  \item  $p$ starts at $(-1,0)$;
  \item each step of $p$ goes to the right (along $(1,0)$) or down (along $(0,-1)$);
  \item $p$ terminates at a point of the form $(2n-1-x,-x)$, $x \in \{0,1, \ldots , n\}$;
  \item $p$ separates all boxes corresponding to elements of $D_{2} \cap supp(\mathfrak{i})$ from boxes corresponding to elements of $D_{2} \setminus supp(\mathfrak{i})$.
  \item  $q$ starts at $(2n,1)$;
  \item each step of $q$ goes to the left (along $(-1,0)$) or up (along $(0,1)$);
  \item $q$ terminates at a point of the form $(y,y+1)$, $y \in \{0,1, \ldots , n\}$;
  \item $q$ separates all boxes corresponding to elements of $D_{3} \cap supp(\mathfrak{i})$ from boxes corresponding to elements of $D_{3} \setminus supp(\mathfrak{i})$.
 \end{itemize}

By substituting each $(1,0)$-step in $p$ with $r$ and each $(0,-1)$-step in $p$ with $f$ we obtain a $B_{n}$-path $p(\mathfrak{i})$, and by substituting
each $(-1,0)$-step in $q$ with $r$ and each $(0,1)$-step in $q$ with $f$ we obtain another $B_{n}$-path $q(\mathfrak{i})$. Thus, with the basic ideal $\mathfrak{i}$ we have associated
 a pair of $B_{n}$-paths $(p(\mathfrak{i}),q(\mathfrak{i}))$. It is also clear that different basic ideals give rise to different pairs of paths.

However, as mentioned before, not every pair $(p,q)$ corresponds to a basic ideal. All pairs of the form $(p(\mathfrak{i}),q(\mathfrak{i}))$ for some basic ideal $\mathfrak{i}$ are called \emph{admissible}.
 For example, the pair $(rrrfrfrr,rrrrrrff)$
 is admissible; it is precisely the pair from the previous picture.

We now specify the conditions for a pair $(p,q)$ of $B_{n}$-paths to be admissible.

For $1 \leq i < j \leq 2n$ define  $\mathfrak{B}_{n}(i,j)$ as the set of all $B_{n}$-paths in which the first and the second
 occurrences of $f$ (reading the word from left to right) are at the positions $i$ and $j$.
 Also, for $1 \leq i \leq 2n$ let  $\mathfrak{B}_{n}(i)$ be the set of paths where the first occurrence of $f$ is at position $i$. 
For example $rrrfrfff \in \mathfrak{B}_{4}(4,6) \subset \mathfrak{B}_{4}(4)$ and $rrrfff \in \mathfrak{B}_{3}(4,5) \subset \mathfrak{B}_{3}(4)$.
 Note that $B_{n}$ paths also are defined for $n=1$, and we have $|\mathfrak{B}_{1}(2)|=|\mathfrak{B}_{1}(3)|=1$. We shall also abuse notation
 by defining $\{rr \cdots rr\} =  \mathfrak{B}_{n}(2n+1)$, that is, if a path contains no $f$, its first $f$ can be thought of to be one step after the last letter of the path.
 Similarly, write $\{rr \cdots rr\} = \mathfrak{B}_{n}(2n+1,2n+2)$ and $rr \cdots rfr \cdots rr  \in \mathfrak{B}_{n}(k,2n+1)$ (here the only occurrence of $f$ is at place $k$).

Using the notation above we can state the following criterion for a pair of $B_{n}$-paths $(p,q)$ to be admissible.

\begin{prop}
 Let $p \in \mathfrak{B}_{n}(a,b)$ where $3 \leq a < b \leq 2n$. Then $(p,q)$ is an admissible pair of paths if and only if $q \in \mathfrak{B}_{n}(i,j)$
 where $i \geq 2n+4-b$ and $j \geq 2n+4-a$.
\end{prop}
\begin{proof}
 Let $3 \leq a < b \leq 2n$ and $p \in \mathfrak{B}_{n}(a,b)$. Realize $(p,q)$ as a pair of paths in the diagram described above.
 By Proposition~\ref{pathprop}, the path $p$ specifies a coideal of the poset $D_{2}$ and $q$ specifies a coideal of the poset $D_{3}$, so we need only consider $\prec$-arrows between the two triangles.
 Such arrows originate from the two upper rows of the lower diagram.
 The root corresponding to the box whose northwest corner is in $(a-2,0)$ is in the coideal, and it $( \prec )$-covers                                                         
 the root corresponding to the box with northwest corner in $(a-2,3)$, so this latter root must
 also belong to the coideal. Similarly, since the root corresponding to the box with northwest corner $(b-3,-1)$ is in the coideal, the root corresponding to the box with
 northwest corner in $(b-3,2)$ must also be in the coideal. Thus, for $p$ as above, $(p,q)$ is admissible if and only if
 the two boxes with northwest corners $(a-2,3)$ and $(b-3,2)$ respectively lie to the northwest
 of the path $q$. This is equivalent to $q \in \mathfrak{B}_{n}(i,j)$ with $i \geq 2n+4-b$ and $j \geq 2n+4-a$.
\end{proof}

We illustrate the proposition above with an example.

\begin{example}
Let $n=6$ and let $p=rrrrfrrfrffr \in \mathfrak{B}_{6}(5,8)$. The path $p$ is drawn in red in the bottom triangle of the following picture.
\[
\begin{tikzpicture}
  \draw (0,0) -- (1,0) -- (2,0) -- (3,0) -- (4,0) -- (5,0) -- (6,0) -- (7,0) -- (8,0) -- (9,0) -- (10,0) -- (11,0)
 -- (11,-1) -- (10,-1) -- (10,-2) -- (9,-2) -- (9,-3) -- (8,-3) -- (8,-4) -- (7,-4) -- (7,-5) -- (6,-5) -- (6,-6) -- (5,-6) -- (5,-5) -- (4,-5) -- (4,-4)
 -- (3,-4)  -- (3,-3) -- (2,-3) -- (2,-2) -- (1,-2) -- (1,-1) -- (0,-1) --cycle;

  \draw (0,1) -- (1,1) -- (2,1) -- (3,1) -- (4,1) -- (5,1) -- (6,1) -- (7,1) -- (8,1) -- (9,1) -- (10,1) -- (11,1)
 -- (11,2) -- (10,2) -- (10,3) -- (9,3) -- (9,4) -- (8,4) -- (8,5) -- (7,5) -- (7,6) -- (6,6) -- (6,7) -- (5,7) -- (5,6) -- (4,6) -- (4,5)
 -- (3,5)  -- (3,4) -- (2,4) -- (2,3) -- (1,3) -- (1,2) -- (0,2) --cycle;

  \draw[very thick,color=red]  (-1,0) -- (0,0) -- (1,0) -- (2,0) -- (3,0) -- (3,-1) -- (4,-1) -- (5,-1) -- (5,-2) -- (6,-2) 
     -- (6,-3) -- (6,-4) -- (7,-4);
  \draw[very thick,color=blue]  (12,1) -- (4,1) -- (4,2) -- (1,2);

  \draw[color=black,dashed]  (4,2) -- (3,2) -- (3,3) -- (4,3) -- cycle;
  \draw[color=black, dashed]  (6,1) -- (5,1) -- (5,2) -- (6,2) -- cycle;
  \draw[color=black,dashed]  (4,-1) -- (3,-1) -- (3,0) -- (4,0) -- cycle;
  \draw[color=black, dashed]  (6,-2) -- (5,-2) -- (5,-1) -- (6,-1) -- cycle;

  \node (a) at (3.5,-0.5) {(1)};
  \node (b) at (3.5,2.5) {(1')};
  \node (c) at (5.5,-1.5) {(2)};
  \node (d) at (5.5,1.5) {(2')};
\end{tikzpicture}
\]
The box marked by $(1)$ is linked to the box marked $(1')$ linked via the $\prec$-relation, and the box marked by $(2)$ is linked to the box marked $(2')$ linked via the $\prec$-relation.
 Hence $q$ must be chosen such that the boxes marked $(1')$ and $(2')$ are both to the northeast of the path $q$.
 In other words, $q \in \mathfrak{B}_{6}(i,j)$ where $i \geq 2n+4-b=8$ and $j \geq 2n+4-a=11$. The smallest coideal allowed for this choice of the path $p$ is of course given by
 $q=(rrrrrrrfrrff)$, and the largest ideal corresponds to when $q=rrrrrrrrrrrr$. All paths ``between'' these also give admissible pairs $(p,q)$. One such choice, $q=rrrrrrrrfrrr$,
 is drawn in blue color in the picture.
\end{example}

\begin{cor}
 For $p \in \mathfrak{B}_{n}(a,b)$ where $3 \leq a < b \leq 2n$, the number of $q$'s such that $(p,q)$ is admissible is given by
 \[ \sum_{i \geq 2n+4-b} \,\, \sum_{j \geq 2n+4-a} |\mathfrak{B}_{n}(i,j)|.\]
\end{cor}

Comparing with the picture of the poset though, we note that the argument above does not make sense when $a=2$ or when $b=2n+1$, so unfortunately we need
 to consider a few special cases as well.

\section{Counting paths}
\label{s6}
In this section we prove Theorem~\ref{thmmain}.
We start with the following lemma.

\begin{lemma}
\label{lemma1}
For $n \geq 2$ and $2 \leq i < j \leq 2n+1$ we have
\[|\mathfrak{B}_{n}(i,j)| = |\mathfrak{B}_{n-1}(j-2)|. \]
\end{lemma}

\begin{proof}
For such $i,j$ and $n$, define a map $\phi:\mathfrak{B}_{n}(i,j) \rightarrow \mathfrak{B}_{n-1}(j-2)$ which just
 deletes the first (leftmost) $r$ and the first $f$ from a given path.
 It is easy to check that $\phi(p) \in \mathfrak{B}_{n-1}(j-2)$ and that $\phi$ is a bijection (its inverse just appends an $r$ to the
 left and inserts an $f$ at position $i$). Hence the sets have the same cardinality.
\end{proof}

The numbers $|\mathfrak{B}_{n}(i)|$ are given explicitly by the following lemma.
\begin{lemma}
\label{lemmax}
For $n \geq 2$ and $1 \leq i \leq 2n$ we have

\[|\mathfrak{B}_{n}(i)|=
\begin{cases}
\displaystyle \sum_{k=1}^{n-1} 2^{k} \Bigg( \binom{2n-i-k-1}{n-i} - \binom{2n-i-k-1}{n-i-k} \Bigg) ,  & \text{ if  $i \leq n$;}\\ 
\\
2^{2n-i} ,       & \text{ if $i > n$.}
\end{cases}
\]
\end{lemma}
\begin{proof}

Each path of $\mathfrak{B}_{n}(i)$ can be specified by first choosing a Dyck-path starting with $r^{i-1}f$, and then replacing
 its tail consisting of some number of $f$'s by any sequence of $r$ and $f$ of the same length. Baur and Mazorchuk~\cite[Proposition~13]{BaMa} showed that the number of Dyck-paths of semilength $n$ with the first peak at height $i$
 and the last peak at height $k$ is given by the expression \[ \binom{2n-i-k-2}{n-i-1} - \binom{2n-i-k-2}{n-i-k-1}. \] Since a Dyck-path having the last peak at height $k$ ends with $f^{k}$, it gives rise to $2^{k}$
 possible $B$-paths. Summing over $k$ we obtain the number of all such $B$-paths. When $i > n$, this argument degenerates, but a path starting with $r^{i-1}f$ with $i > n$ can clearly be completed by choosing the remaining symbols
 arbitrarily, which gives $2^{2n-i}$ possibilities. Here is a geometric illustration.
\[
\begin{tikzpicture}
  \draw (0,0) -- (1,0) -- (2,0) -- (3,0) -- (4,0) -- (5,0) -- (6,0) -- (7,0) -- (8,0) -- (9,0) 
 -- (9,-1) -- (8,-1) -- (8,-2) -- (7,-2) -- (7,-3) -- (6,-3) -- (6,-4) -- (5,-4) -- (5,-5) -- (4,-5) -- (4,-4)
 -- (3,-4)  -- (3,-3) -- (2,-3) -- (2,-2) -- (1,-2) -- (1,-1) -- (0,-1) --cycle;
  \draw (7,0) -- (8,0) -- (9,0); 
  \draw[very thick,color=red] (-1,0) -- (0,0) -- (0,-1) -- (1,-1) -- (2,-1) -- (3,-1) -- (3,-2) -- (4,-2) -- (4,-3) -- (4,-4) -- (4,-5);
  \draw[very thick,color=red, dashed] (4,-2)--(5,-2) -- (5,-3) -- (6,-3);
\end{tikzpicture}
\]
The red path above corresponds to the Dyck-path $rfrrrfrfff$. It can be modified to a $B$-path by replacing the three $f$'s at the end by any three letter word on $\{r,f\}$.
 There are $2^{3}$ such choices, the choice $rfr$ gives the $B$-path $rfrrrfrrfr$ which is displayed as a dashed path in the figure.
\end{proof}

To proceed we need a number of lemmas.

\subsection{Lemmata}
\begin{lemma}
\label{lemmac2}
For integers $n\geq 0$ we have \[\sum_{k=0}^{n} k 2^{n-k} = 2^{n+1}-n-2\]
\end{lemma}
\begin{proof}
 Induction on $n$.
\end{proof}

\begin{lemma}
\label{lemmac1}
For integers $n\geq 0$ we have \[\sum_{k=0}^{n} k^{2} 2^{n-k} = 6\cdot 2^{n}-n^{2}-4n-6\]
\end{lemma}
\begin{proof}
 Induction on $n$.
\end{proof}

\begin{lemma}
\label{lemmaCK}
For integers $ 0 \leq m \leq n$ we have \[\sum_{s=m}^{n} \binom{s}{m} = \binom{n+1}{m+1}\]
\end{lemma}
\begin{proof}
This follows from Pascal's rule.
\end{proof}

\begin{lemma}
\label{lemmaCK2}
For all positive integers $A$ and $B$ with $A \geq B$, we have \[\sum_{\ell=0}^{A-B} 2^{\ell} \binom{A-\ell}{B} = \sum_{s=0}^{A-B}\binom{A+1}{s}\]
\end{lemma}
\begin{proof}
We first convert the sum on the left into standard hypergeometric notation, as given by 
\[\setlength\arraycolsep{1pt}
{}_p F_q\left[\begin{matrix}a_{1}, \ldots ,a_{p}\\
b_{1}, \ldots ,b_{q}\end{matrix};z\right] = \sum_{\ell = 0}^{\infty} \frac{(a_{1})_{\ell}, \ldots ,(a_{p})_{\ell}}{\ell !  (b_{1})_{\ell}, \ldots ,(b_{q})_{\ell}} z^{\ell},\]
 where the Pochhammer symbol $(\alpha)_{\ell}$ is defined by \[(\alpha)_{\ell} := \alpha(\alpha+1) \cdots (\alpha+\ell-1)\] for $\ell \geq 1$, and $(\alpha)_{0}:=1$.

In this notation we have \[\setlength\arraycolsep{1pt}\sum_{\ell=0}^{A-B} 2^{\ell} \binom{A-\ell}{B} = \binom{A}{B}{}_2 F_1\left[\begin{matrix}1,-A+B\\
-A\end{matrix};2\right].\]

We now apply the following transformation formula for hypergeometric functions, see Slater~\cite[(1.8.10)]{LJS}.
\[ \setlength\arraycolsep{1pt}
{}_2 F_1\left[\begin{matrix}a,-N\\
c\end{matrix};z\right] = (1-z)^{N} \frac{(a)_{N}}{(c)_{N}} {}_2 F_1\left[\begin{matrix}-N,c-a\\
1-a-N\end{matrix};\frac{1}{1-z}\right],\] where $N$ is a nonnegative integer.

Choosing $a=1 + \varepsilon$, $N=A-B$, $c=-A+\varepsilon$ and $z=2$ gives the formula
\[ \setlength\arraycolsep{1pt}
{}_2 F_1\left[\begin{matrix}1,-A+B\\
-A\end{matrix};2\right] = (-1)^{A-B} \lim_{\varepsilon \rightarrow 0}\frac{(1+\varepsilon)_{A-B}}{(-A+\varepsilon)_{A-B}} {}_2 F_1\left[\begin{matrix}-A+B,-A-1\\
-A+B-\varepsilon\end{matrix};-1\right].\]

Applying this formula to our previous expression we obtain

\begin{align*}
\setlength\arraycolsep{1pt}\sum_{\ell=0}^{A-B}& 2^{\ell} \binom{A-\ell}{B} = \binom{A}{B} \lim_{\varepsilon \rightarrow 0}\frac{(1+\varepsilon)_{A-B}}{(B+1-\varepsilon)_{A-B}} {}_2 F_1\left[\begin{matrix}-A+B,-A-1\\
-A+B-\varepsilon\end{matrix};-1\right] \\
&= \sum_{s=0}^{A-B}\frac{(-A-1)_{s}}{s!}(-1)^{s}\\
&= \sum_{s=0}^{A-B} \binom{A+1}{s}.
\end{align*}
\end{proof}

We are now ready to prove Theorem~\ref{thmmain}. The proof will occupy the remainder of the paper, so we shall divide it into a number of parts.

\emph{Layout of the proof:} The number $\tilde{b}_{n}$ equals the number of admissible pairs of $B_{n}$-paths. We shall
 count these by counting the numbers of admissible pairs $(p,q)$ where
 $p \in \mathfrak{B}(a,b)$ and $q \in \mathfrak{B}(i,j)$ and then taking the sum over all relevant such integer quadruples $(a,b,i,j)$.
 We partition these quadruples into four disjoint classes and we proceed by separately computing the number of admissible pairs belonging to each class.
 The rest of the proof is a simplification of the obtained expression.

\subsection{Case I}
Consider those $(a,b,i,j)$ where $a \geq 3$, $4 \leq b \leq 2n$ and $i \leq 2n + 1$. This is the most general case. For each such $a$ and $b$,
 we have $|\mathfrak{B}_{n}(a,b)| = |\mathfrak{B}_{n-1}(b-2)|$ choices for $p$, and for each such choice we choose $j$ between $(2n+4-a)$ and $2n+1$, and for such $a,b,j$ we have
 $(b+j-2n-4)$ ways to choose $i$, and then $|\mathfrak{B}_{n}(i,j)|=|\mathfrak{B}_{n-1}(j-2)|$ ways to complete the path. So for this case, the number
 of admissible paths can be written

\[
s_{1} = \sum_{b=4}^{2n} |\mathfrak{B}_{n-1}(b-2)| \sum_{a=3}^{b-1} \,\, \sum_{j=2n+4-a}^{2n+1} (b+j-2n-4)  |\mathfrak{B}_{n-1}(j-2)|.
\]
Here is an attempt of illustration: when $n=4$, the $p$'s for which the number of $(p,q)$ are counted in this case are all drawn on top of each other in the picture below.

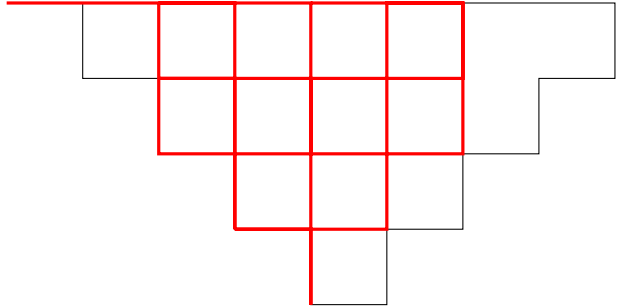
\[
\begin{tikzpicture}
  \draw (0,0) -- (1,0) -- (2,0) -- (3,0) -- (4,0) -- (5,0) -- (6,0) -- (7,0)--
      (7,-1) -- (6,-1) -- (6,-2) -- (5,-2) -- (5,-3) -- (4,-3) -- (4,-4) -- (3,-4) -- (3,-3)  
       -- (2,-3) -- (2,-2) -- (1,-2) -- (1,-1) -- (0,-1) --cycle;
  \draw[very thick,color=red] (-1,0) -- (1,0) -- (1,-2) -- (3,-2) -- (3,0) -- (1,0) -- (2,0) -- (2,-3) -- (2,-1)-- (1,-1) -- (3,-1)
       -- (3,-4) -- (3,-3) -- (2,-3) -- (4,-3) -- (4,0) -- (5,0) -- (5,-2) -- (3,-2) -- (3,-1) -- (5,-1) -- (5,0) -- (3,0) ;
\end{tikzpicture}
\]

\subsection{Case II}
Next consider the case when $a \geq 3$, $b = 2n+1$ and $i \leq 2n + 1$, that is, the case when $p$ only contains one $f$.
 The path $p$ is then determined by the choice of $a$, and the path $q$ can be chosen as in Case I, except that there are now less restriction on $i$:
 any number $2 \leq i \leq j-1$ will do, so there are $(j-2)$ choices.

\[
s_{2} =  \sum_{a=3}^{2n} \,\,  \sum_{j=2n+4-a}^{2n+1} (j-2)  |\mathfrak{B}_{n-1}(j-2)|
\]

Again, an illustration for $n=4$:
\[
\begin{tikzpicture}
  \draw (0,0) -- (1,0) -- (2,0) -- (3,0) -- (4,0) -- (5,0) -- (6,0) -- (7,0)--
      (7,-1) -- (6,-1) -- (6,-2) -- (5,-2) -- (5,-3) -- (4,-3) -- (4,-4) -- (3,-4) -- (3,-3)  
       -- (2,-3) -- (2,-2) -- (1,-2) -- (1,-1) -- (0,-1) --cycle;
  \draw[very thick,color=red] (-1,0) -- (6,0) -- (6,-1) -- (1,-1) -- (1,0) -- (2,0) -- (2,-1) 
    -- (3,-1) -- (3,0) -- (4,0) -- (4,-1) -- (5,-1) -- (5,0);
\end{tikzpicture}
\]

\subsection{Case III}
Next consider the case when $a=2$, $b \leq 2n$ and $i \leq 2n + 1$ (which corresponds to the paths $p$ starting with $rf$).
 When $b$ is chosen, and the number of $p$ is counted, the possibilities for $q$
 are very restricted. We have one choice for $j$, and $(b-3)$ choices for $i$, that is

\[
s_{3} = \sum_{b=4}^{2n} |\mathfrak{B}_{n-1}(b-2)| (b-3).
\]
This is illustrated by the following picture.
\[
\begin{tikzpicture}
  \draw (0,0) -- (1,0) -- (2,0) -- (3,0) -- (4,0) -- (5,0) -- (6,0) -- (7,0)--
      (7,-1) -- (6,-1) -- (6,-2) -- (5,-2) -- (5,-3) -- (4,-3) -- (4,-4) -- (3,-4) -- (3,-3)  
       -- (2,-3) -- (2,-2) -- (1,-2) -- (1,-1) -- (0,-1) --cycle;
  \draw[very thick,color=red] (-1,0) -- (0,0) -- (0,-1) -- (1,-1) -- (1,-2) -- (3,-2) -- (3,-1)  -- (1,-1) -- (2,-1) -- (2,-3)
                   -- (3,-3) -- (3,-4) -- (3,-1) -- (5,-1) -- (5,-2) -- (3,-2) -- (4,-2) -- (4,-1) -- (4,-3) -- (3,-3);
\end{tikzpicture}
\]

\subsection{Case IV}
This case considers all remaining admissible pairs of paths. When $a=2$, $b=2n+1$ and $i \leq 2n + 1$, the path $p$ is fixed and $q$ is determined by its parameter i for which there are $(2n-1)$ choices.

Everything up to now has been restricted by the condition $i \leq 2n + 1$. It remains to consider the case when $q=rr \cdots rr$. In this case
 we have $\binom{2n}{n}$ possibilities for $p$.

Similarly, we have to consider the case $p = rr \cdots rr$, where we also have $\binom{2n}{n}$ possibilities for $q$.
 Here the single case $(p,q)=(rr \cdots rr,rr \cdots rr)$
 has to be subtracted though, in order not to count it twice.
\[s_{4} = (2n-1) + \binom{2n}{n} + \binom{2n}{n} - 1\]

\subsection{Simplification of the total sum}
All in all, in summing up we have $\tilde{b}_{n} = s_{1} + s_{2} + s_{3} + s_{4}$ or explicitly

\begin{align*}
\tilde{b}_{n} &= \sum_{b=4}^{2n} |\mathfrak{B}_{n-1}(b-2)| \,\, \sum_{a=3}^{b-1} \sum_{j=2n+4-a}^{2n+1} (b+j-2n-4)  |\mathfrak{B}_{n-1}(j-2)|\\
 &+ \sum_{a=3}^{2n} \sum_{j=2n+4-a}^{2n+1} (j-2)  |\mathfrak{B}_{n-1}(j-2)| \\
 &+ \sum_{b=4}^{2n} |\mathfrak{B}_{n-1}(b-2)| (b-3)\\
 &+ 2\binom{2n}{n} + (2n -2).
\end{align*}

In the first and second term (row) the variable $a$ only determines the index of a sum, so we can eliminate this variable.

\begin{align*}
\tilde{b}_{n} &= \sum_{b=4}^{2n} |\mathfrak{B}_{n-1}(b-2)| \sum_{j=2n-b+5}^{2n+1} (b+j-2n-4)^{2}  |\mathfrak{B}_{n-1}(j-2)|\\
 &+ \sum_{j=4}^{2n+1} (j-3)(j-2)  |\mathfrak{B}_{n-1}(j-2)| \\
 &+ \sum_{b=4}^{2n} |\mathfrak{B}_{n-1}(b-2)| (b-3)\\
 &+ 2\binom{2n}{n} + (2n -2).
\end{align*}

We next replace $j$ by $j-2n+b-3$ as index variable in the first expression, and the third and first terms are factored together. This gives

\begin{align*}
\tilde{b}_{n} &= \sum_{b=4}^{2n} |\mathfrak{B}_{n-1}(b-2)| \Bigg( (b-3) + \sum_{j=2}^{b-2} (j-1)^{2}  |\mathfrak{B}_{n-1}(j+2n-b+1)| \Bigg) \\
 &+ \Bigg( \sum_{j=4}^{2n+1} (j-3)(j-2)  |\mathfrak{B}_{n-1}(j-2)| \Bigg) + 2\binom{2n}{n} + (2n -2). 
\end{align*}

In the sum going from $j=2$ to $b-2$ we can sum from $1$ instead since the new first term will be $0$. The last term is $(b-3)^{2}$, and $(b-3)^{2} + (b-3)=(b-3)(b-2)$ so  

\begin{align*}
\tilde{b}_{n} &= \sum_{b=4}^{2n} |\mathfrak{B}_{n-1}(b-2)| \Bigg( (b-3)(b-2) + \sum_{j=1}^{b-3} (j-1)^{2}  |\mathfrak{B}_{n-1}(j+2n-b+1)| \Bigg) \\
 &+ \Bigg( \sum_{j=4}^{2n+1} (j-3)(j-2)  |\mathfrak{B}_{n-1}(j-2)| \Bigg) + 2\binom{2n}{n} + (2n -2). 
\end{align*}

Similarly we take out the last term of the last sum and write

\begin{align*}
\tilde{b}_{n} &= \sum_{b=4}^{2n} |\mathfrak{B}_{n-1}(b-2)| \Bigg( (b-3)(b-2) + \sum_{j=1}^{b-3} (j-1)^{2}  |\mathfrak{B}_{n-1}(j+2n-b+1)| \Bigg) \\
 &+ \Bigg( \sum_{j=4}^{2n} (j-3)(j-2)  |\mathfrak{B}_{n-1}(j-2)| \Bigg) + (2n-2)(2n-1) + 2\binom{2n}{n} + (2n -2). 
\end{align*}

Now the last sum over $j$ can be absorbed into the first term, and we can factor two terms at the end to obtain

\begin{align*}
\tilde{b}_{n} &= \sum_{b=4}^{2n} |\mathfrak{B}_{n-1}(b-2)| \Bigg( 2(b-3)(b-2) + \sum_{j=1}^{b-3} (j-1)^{2}  |\mathfrak{B}_{n-1}(j+2n-b+1)| \Bigg) \\
 &+ 4n(n-1) + 2\binom{2n}{n}. 
\end{align*}

Next we want to apply Lemma \ref{lemmax} to rewrite the functions $\mathfrak{B}$ in a more explicit form. Since $\mathfrak{B}$ is a piecewise defined function, we first split up our sums in a corresponding way.

\begin{align*}
\tilde{b}_{n} &= \sum_{b=4}^{n+1} |\mathfrak{B}_{n-1}(b-2)| \Bigg( 2(b-3)(b-2) + \sum_{j=1}^{b-3} (j-1)^{2}  |\mathfrak{B}_{n-1}(j+2n-b+1)| \Bigg) \\
 & +\sum_{b=n+2}^{2n} |\mathfrak{B}_{n-1}(b-2)| \Bigg( 2(b-3)(b-2) + \sum_{j=1}^{b-3} (j-1)^{2}  |\mathfrak{B}_{n-1}(j+2n-b+1)| \Bigg) \\
 & + 4n(n-1) + 2\binom{2n}{n} \\
 &= \sum_{b=4}^{n+1} |\mathfrak{B}_{n-1}(b-2)| \Bigg( 2(b-3)(b-2) + \sum_{j=1}^{b-3} (j-1)^{2}  |\mathfrak{B}_{n-1}(j+2n-b+1)| \Bigg) \\
 & +\sum_{b=n+2}^{2n} |\mathfrak{B}_{n-1}(b-2)| \Bigg( 2(b-3)(b-2) + \bigg( \sum_{j=1}^{b-n-2} (j-1)^{2}  |\mathfrak{B}_{n-1}(j+2n-b+1)|\\
 & \qquad + \sum_{j=b-n-1}^{b-3} (j-1)^{2}  |\mathfrak{B}_{n-1}(j+2n-b+1)| \bigg) \Bigg) \\
 &+ 4n(n-1) + 2\binom{2n}{n} \\
\end{align*}
Now the sums are partitioned properly and we can apply Lemma \ref{lemmax}. To avoid using a lot of binomial coefficients we temporarily introduce the notation
\[D(n,i,j):=\binom{2n-i-j-2}{n-i-1} - \binom{2n-i-j-2}{n-i-j-1}.\]
Lemma~\ref{lemmax} now gives

\begin{align*}
\tilde{b}_{n} &= \sum_{b=4}^{n+1} \sum_{k=1}^{n-2} 2^{k} D(n-1,b-3,k) \Bigg( 2(b-3)(b-2) + \sum_{j=1}^{b-3} (j-1)^{2}  2^{b-j-3} \Bigg) \\
 & +\sum_{b=n+2}^{2n} 2^{2n-b} \Bigg( 2(b-3)(b-2) + \bigg( \sum_{j=1}^{b-n-2} (j-1)^{2}  \sum_{k=1}^{n-2} 2^{k} D(n-1,2n+j-b,k)\\
 & \qquad + \sum_{j=b-n-1}^{b-3} (j-1)^{2}  2^{b-j-3} \bigg) \Bigg) \\
 &+ 4n(n-1) + 2\binom{2n}{n} \\
\end{align*}

Changing summation variables such that both sums over $b$ goes up to $n-2$ we obtain

\begin{align*}
\tilde{b}_{n} &= \sum_{b=1}^{n-2}  \sum_{k=1}^{n-2} 2^{k} D(n-1,b,k)  \Bigg( 2b(b+1) + \sum_{j=1}^{b} (j-1)^{2}  2^{b-j} \Bigg) \\
 & +\sum_{b=0}^{n-2} 2^{n-b-2} \Bigg( 2(b+n-1)(b+n) + \bigg( \sum_{j=1}^{b} (j-1)^{2}  \sum_{k=1}^{n-2} 2^{k} D(n-1,n+j-b-2,k)\\
 & \qquad + \sum_{j=b+1}^{b+n-1} (j-1)^{2}  2^{b+n-j-1} \bigg) \Bigg) \\
 &+ 4n(n-1) + 2\binom{2n}{n} \\
\end{align*}

Using Lemma \ref{lemmac1} (and changing variables) we can simplify the following two expressions.

\[\sum_{j=1}^{b} (j-1)^{2}  2^{b-j} = 3 \cdot 2^{b}-b^{2}-2b-3\]
\[\sum_{j=b+1}^{b+n-1} (j-1)^{2}  2^{b+n-j-1} = 2^{n-1}(b^{2}+2b+3)-n^{2}-b^{2}-2bn-2\]
Inserting this into our expression for $\tilde{b}_{n}$ we have 

\begin{align*}
\tilde{b}_{n} &= \sum_{b=1}^{n-2}  \sum_{k=1}^{n-2} 2^{k} D(n-1,b,k)  ( 3 \cdot 2^{b}+b^{2}-3 ) \\
 & +\sum_{b=0}^{n-2} 2^{n-b-2} \Bigg( 2(b+n-1)(b+n) + \bigg( \sum_{j=1}^{b} (j-1)^{2}  \sum_{k=1}^{n-2} 2^{k} D(n-1,n+j-b-2,k)\\
 & \qquad + 2^{n-1}(b^{2}+2b+3)-n^{2}-b^{2}-2bn-2 \bigg) \Bigg) \\
 &+ 4n(n-1) + 2\binom{2n}{n} \\
\end{align*}

This can be rewritten as
\begin{align*}
\tilde{b}_{n} &= \sum_{b=1}^{n-2}  \sum_{k=1}^{n-2} 2^{k} D(n-1,b,k)  ( 3 \cdot 2^{b}+b^{2}-3 ) \\
 & +\sum_{b=0}^{n-2} 2^{n-b-2} \Bigg( \sum_{j=1}^{b} (j-1)^{2}  \sum_{k=1}^{n-2} 2^{k} D(n-1,n+j-b-2,k) \Bigg) \\
 & + \sum_{b=0}^{n-2} \big( 2^{2n-b-3}(b^{2}+2b+3)+2^{n-b-2}(n^{2}+b^{2}+2bn-2b-2n-2) \big) \\
 &+ 4n(n-1) + 2\binom{2n}{n} \\
\end{align*}

Here the last sum over $b$ equals $2^{2n+1}-2^{n}(n+3)-4n(n-1)$ which is easily proved by application of Lemma~\ref{lemmac2} and Lemma~\ref{lemmac1}. Canceling the term $4n(n-1)$ we have

\begin{align*}
\tilde{b}_{n} &= \sum_{b=1}^{n-2}  \sum_{k=1}^{n-2} 2^{k} D(n-1,b,k)  ( 3 \cdot 2^{b}+b^{2}-3 ) \\
 & +\sum_{b=0}^{n-2} 2^{n-b-2} \Bigg(\sum_{j=1}^{b} (j-1)^{2}  \sum_{k=1}^{n-2} 2^{k} D(n-1,n+j-b-2,k) \Bigg) \\
 &+ 2^{2n+1}-2^{n}(n+3) + 2\binom{2n}{n} \\
\end{align*}
 
Expanding our function $D$ we more explicitly have

\begin{align*}
\tilde{b}_{n} & = \sum_{b=1}^{n-2}  \sum_{k=1}^{n-2} 2^{k} \bigg( \binom{2n-b-k-4}{n-b-2} - \binom{2n-b-k-4}{n-b-k-2}\bigg) ( 3 \cdot 2^{b}+b^{2}-3 ) \\
 & +\sum_{b=0}^{n-2} 2^{n-b-2} \Bigg(\sum_{j=1}^{b} (j-1)^{2}  \sum_{k=1}^{n-2} 2^{k} \bigg( \binom{n+b-j-k-2}{b-j}-\binom{n+b-j-k-2}{b-j-k} \bigg) \Bigg) \\
 &+ 2^{2n+1}-2^{n}(n+3) + 2\binom{2n}{n} \\
\end{align*}

Changing order of summation to merge the two sums over $k$, and rewriting some binomial coefficients we obtain

\begin{align*}
\tag{1} 
 \tilde{b}_{n} & = \sum_{k=1}^{n-2}2^{k} \sum_{b=1}^{n-2} \Bigg( (3 \cdot 2^{b} +b^{2}-3) \bigg( \binom{2n-b-k-4}{n-b-2} - \binom{2n-b-k-4}{n-2} \bigg) \\
          &+ 2^{n-b-2} \sum_{j=1}^{b}(j-1)^{2} \bigg( \binom{n+b-j-k-2}{n-k-2}-\binom{n+b-j-k-2}{n-2} \bigg) \Bigg)\\
         & -2^{n}(n+3) + 2^{2n+1} + 2\binom{2n}{n} \\
\end{align*}

\subsection{Further simplification}
The remaining simplification of the above expression, including the proof of Lemma \ref{lemmaCK2}
 above, is due to Christian Krattenthaler.

We continue by simplifying the sum over $j$ in $(1)$. For this purpose we write 
\[(j-1)^{2} = (n+b-j-k)(n+b-j-k-1)-(n+b-j-k-1)(2n+2b-2k-3)+(n+b-k-2)^{2}\]
and use this to simplify the binomial coefficients. Thus the sum over $j$ becomes

\begin{align*}
&(n-k)(n-k-1) \sum_{j=1}^{b}\binom{n+b-j-k}{n-k} -n(n-1) \sum_{j=1}^{b}\binom{n+b-j-k}{n} \\
 &-(2n+2b-2k-3)\Big( (n-k-1) \sum_{j=1}^{b} \binom{n+b-j-k-1}{n-k-1}\\
 & \qquad \qquad \qquad \qquad \qquad -(n-1)\sum_{j=1}^{b} \binom{n+b-j-k-1}{n-1} \Big)\\
 &+ (n+b-k-2)^{2} \Big(\sum_{j=1}^{b}\binom{n+b-j-k-2}{n-k-2}\\
 & \qquad \qquad \qquad \qquad \qquad-\sum_{j=1}^{b}\binom{n+b-j-k-2}{n-2}\Big) .\\ 
\end{align*}

Changing variables in each sum and then removing zero-terms we obtain

\begin{align*}
(n-k)(n-k-1) &\sum_{j=n-k}^{n+b-k-1}\binom{j}{n-k} -n(n-1) \sum_{j=n-k}^{n+b-k-1}\binom{j}{n} \\
 -(2n+2b-2k-3)\Big( (n-k-1) &\sum_{j=n-k-1}^{n+b-k-2} \binom{j}{n-k-1} -(n-1)\sum_{j=n-k-1}^{n+b-k-2} \binom{j}{n-1} \Big)\\
 + (n+b-k-2)^{2} \Big(&\sum_{j=n-k-2}^{n+b-k-3}\binom{j}{n-k-2}-\sum_{j=n-k-2}^{n+b-k-3}\binom{j}{n-2}\Big) \\
 = (n-k)(n-k-1) &\sum_{j=n-k}^{n+b-k-1}\binom{j}{n-k} -n(n-1) \sum_{j=n}^{n+b-k-1}\binom{j}{n} \\
 -(2n+2b-2k-3)\Big( (n-k-1) &\sum_{j=n-k-1}^{n+b-k-2} \binom{j}{n-k-1} -(n-1)\sum_{j=n-1}^{n+b-k-2}\binom{j}{n-1} \Big)\\
 + (n+b-k-2)^{2} \Big(&\sum_{j=n-k-2}^{n+b-k-3}\binom{j}{n-k-2}-\sum_{j=n-2}^{n+b-k-3}\binom{j}{n-2}\Big) \\
\end{align*}

to which we apply Lemma \ref{lemmaCK} six times. This gives

\begin{multline*}
(n-k)(n-k-1) \binom{n+b-k}{n-k+1} -n(n-1)\binom{n+b-k}{n+1} \\
 -(2n+2b-2k-3)\Bigg( (n-k-1)\binom{n+b-k-1}{n-k} -(n-1)\binom{n+b-k-1}{n} \Bigg)\\
 + (n+b-k-2)^{2} \Bigg(\binom{n+b-k-2}{n-k-1}-\binom{n+b-k-2}{n-1}\Bigg) \\
\end{multline*}

which simplifies to

\begin{align*}
\binom{n+b-k-1}{n-k+1}+\binom{n+b-k-2}{n-k+1}-\binom{n+b-k-1}{n+1}-\binom{n+b-k-2}{n+1}. \tag{2}
\end{align*}
So inserting $(2)$ back into $(1)$ we now have

\begin{align*}
 \tilde{b}_{n} & = \sum_{k=1}^{n-2}2^{k} \sum_{b=1}^{n-2} \Bigg( (3 \cdot 2^{b} +b^{2}-3)(\binom{2n-b-k-4}{n-b-2} - \binom{2n-b-k-4}{n-2}) \\
          & \qquad \qquad \qquad \qquad+ 2^{n-b-2} \Big( \binom{n+b-k-1}{n-k+1}+\binom{n+b-k-2}{n-k+1}\\
& \qquad \qquad \qquad \qquad-\binom{n+b-k-1}{n+1}-\binom{n+b-k-2}{n+1}  \Big) \Bigg)\\
         & -2^{n}(n+3) + 2^{2n+1} + 2\binom{2n}{n}. \tag{3}  \\
\end{align*}
On the second line we reverse the order of the $b$-summation, that is, we replace $b$ by $n-1-b$ which yields
\begin{align*}
 \tilde{b}_{n} & = \sum_{k=1}^{n-2}2^{k} \sum_{b=1}^{n-2} \Bigg( (3 \cdot 2^{b} +b^{2}-3)(\binom{2n-b-k-4}{n-b-2} - \binom{2n-b-k-4}{n-2}) \\
          & \qquad \qquad \qquad \qquad+ 2^{b-1} \Big( \binom{2n-b-k-2}{n-k+1}+\binom{2n-b-k-3}{n-k+1}\\
& \qquad \qquad \qquad \qquad-\binom{2n-b-k-2}{n+1}-\binom{2n-b-k-3}{n+1}  \Big) \Bigg)\\
         & -2^{n}(n+3) + 2^{2n+1} + 2\binom{2n}{n}.  \\
\end{align*}

Introducing a new summation variable $\ell = b+k$ we can write this as

\begin{align*}
 \tilde{b}_{n} & = \sum_{k=1}^{n-2}2^{k} \sum_{b=1}^{n-2}  (b^{2}-3)(\binom{2n-b-k-4}{n-k-2} - \binom{2n-b-k-4}{n-2}) \\
          &+ \sum_{k=1}^{n-2} \sum_{\ell = k+1}^{2n-2} 2^{\ell} \Bigg( 3\binom{2n-\ell-4}{n-k-2}+3\binom{2n- \ell -4}{n-2} \\
& + \frac{1}{2}\Big( \binom{2n-\ell-2}{n-k+1}+\binom{2n-\ell -3}{n-k+1}-\binom{2n-\ell -2}{n+1}-\binom{2n-\ell -3}{n+1}  \Big) \Bigg) \\
         & -2^{n}(n+3) + 2^{2n+1} + 2\binom{2n}{n}.  \\
\end{align*}

Next we apply Lemma \ref{lemmaCK2} to all the sums over $\ell$. This gives

\begin{align*}
 \tilde{b}_{n} & = \sum_{k=1}^{n-2}2^{k} \sum_{b=1}^{n-2}  (b^{2}-3)(\binom{2n-b-k-4}{n-k-2} - \binom{2n-b-k-4}{n-2}) \\
          &+ 3 \sum_{k=1}^{n-2} 2^{k+1}\sum_{s=0}^{n-3} \binom{2n-k-4}{s} - 3\sum_{k=1}^{n-2} 2^{k+1}\sum_{s=0}^{n-k-3} \binom{2n- k -4}{s} \\
& + \sum_{k=1}^{n-2} 2^{k} \sum_{s=0}^{n-4}\binom{2n-k-2}{s} + \sum_{k=1}^{n-2} 2^{k} \sum_{s=0}^{n-5}\binom{2n-k-3}{s}\\
& -\sum_{k=1}^{n-2} 2^{k}\sum_{s=0}^{n-k-4}\binom{2n-k -2}{s} - \sum_{k=1}^{n-2} 2^{k}\sum_{s=0}^{n-k-5}\binom{2n-k -3}{s} \\
         & -2^{n}(n+3) + 2^{2n+1} + 2\binom{2n}{n}.  \tag{4} \\
\end{align*}

To simplify $(4)$ we first note that 
\begin{align*}
(1+1)^{2n-k-2}&=\sum_{s=0}^{2n-k-2}\binom{2n-k-2}{s} = \sum_{s=0}^{n+1}\binom{2n-k-2}{s} + \sum_{s=n+2}^{2n-k-2}\binom{2n-k-2}{s} \\
&= \sum_{s=0}^{n+1}\binom{2n-k-2}{s} + \sum_{s=0}^{n-k-4}\binom{2n-k-2}{n-k-4-s}\\
&= \sum_{s=0}^{n+1}\binom{2n-k-2}{s} + \sum_{s=0}^{n-k-4}\binom{2n-k-2}{s}  \\
\end{align*}
so \[\sum_{s=0}^{n-k-4}\binom{2n-k-2}{s} = 2^{2n-k-2} - \sum_{s=0}^{n+1} \binom{2n-k-2}{s}.\]

Using this, and similar identities, we transform the expression $(4)$ to 

\begin{align*}
 \tilde{b}_{n} & = \sum_{k=1}^{n-2} 2^{k} \sum_{b=1}^{n-2}  (b^{2}-3)(\binom{2n-b-k-4}{n-k-2} - \binom{2n-b-k-4}{n-2}) \\
          &+ 3 \sum_{k=1}^{n-2} 2^{k+1}\sum_{s=0}^{n-3} \binom{2n-k-4}{s} - 3(n-2)2^{2n-3}+3\sum_{k=1}^{n-2} 2^{k+1}\sum_{s=0}^{n-2}\binom{2n- k -4}{s} \\
& + \sum_{k=1}^{n-2} 2^{k} \sum_{s=0}^{n-4}\binom{2n-k-2}{s} + \sum_{k=1}^{n-2} 2^{k} \sum_{s=0}^{n-5}\binom{2n-k-3}{s}\\
& -(n-2)2^{2n-2} + \sum_{k=1}^{n-2} 2^{k}\sum_{s=0}^{n+1}\binom{2n-k -2}{s} -(n-2)2^{2n-3} + \sum_{k=1}^{n-2} 2^{k} \sum_{s=0}^{n+1}\binom{2n-k -3}{s} \\
         & -2^{n}(n+3) + 2^{2n+1} + 2\binom{2n}{n}.  \tag{5} \\
\end{align*}

In the second, third and fourth line we can interchange summations over $k$ and $s$ and apply Lemma \ref{lemmaCK2} again. However, one has to be a bit careful since the sums over $k$ are in fact
 shorter than required by the lemma. For example, the first of these sums can be simplified as follows.

\begin{align*}
  3 \sum_{k=1}^{n-2}& 2^{k+1}\sum_{s=0}^{n-3} \binom{2n-k-4}{s}=3\sum_{s=0}^{n-3}\sum_{k=1}^{n-2} 2^{k+1} \binom{2n-k-4}{s}\\
 &=  3 \sum_{s=0}^{n-3} \Big( \sum_{k=1}^{2n-s-4} 2^{k+1}\binom{2n-k-4}{s} - \sum_{k=n-1}^{2n-s-4} 2^{k+1}\binom{2n-k-4}{s} \Big) \\
 &=  3 \sum_{s=0}^{n-3} \Big( 4 \sum_{k=0}^{2n-s-5}\binom{2n-4}{k} - 2^{n} \sum_{k=0}^{n-s-3}\binom{n-2}{k} \Big) \\
 &= 12 \sum_{k=0}^{n-3} (n-2)\binom{2n-4}{k} + 12 \sum_{k=0}^{n-3}(n-k-2)\binom{2n-4}{n+k-2} \\
 & \qquad - 3 \cdot 2^{n} \sum_{k=0}^{n-3}(n-k-2) \binom{n-2}{k}  \\
 &= 6(n-2) \Big( 2^{2n-4}-\binom{2n-4}{n-2} \Big) + 6(n-2)2^{2n-4} - 3 \cdot 2^{n}(n-2)2^{n-3}\\
 &= 3(n-2)2^{2n-3}-6(n-2)\binom{2n-4}{n-2}.
\end{align*}
We treat the other double sums over $k$ and $s$ in $(5)$ similarly, and after many, many calculations we arrive at the expression

\begin{align*}
 \tilde{b}_{n} & = \sum_{k=1}^{n-2} 2^{k} \sum_{b=1}^{n-2}  (b^{2}-3)(\binom{2n-b-k-4}{n-k-2} - \binom{2n-b-k-4}{n-2}) \\
          &+ (3n+5)2^{2n-2} + (n^{2}-1)2^{n-1} - 4(1-7n+7n^{2}+3n^{3})\frac{(2n-3)!}{(n-2)!(n+1)!}. \tag{6}
\end{align*}

Now we turn to the sum over $b$. By expressing \[b^{2} = (2n-b-k-2)(2n-b-k-3)-(4n-2k-5)(2n-b-k-3)+(2n-k-3)^{2}\]
and proceeding by application of Lemma \ref{lemmaCK} as in the beginning of this section, and then rewriting the binomial coefficients we obtain

\begin{align*}
  \sum_{b=1}^{n-2}&  (b^{2}-3)(\binom{2n-b-k-4}{n-k-2} - \binom{2n-b-k-4}{n-2}) \\
 =& (n-k)(n-k-1)\binom{2n-k-2}{n-k+1} - n(n-1)\binom{2n-k-2}{n+1}\\
 &-(4n-2k-5)(n-k-1)\binom{2n-k-3}{n-k} + (4n-2k-5)(n-1)\binom{2n-k-3}{n}\\
 &+((2n-k-3)^{2}-3)\binom{2n-k-4}{n-k-1} - ((2n-k-3)^{2}-3)\binom{2n-k-4}{n-1}\\
 =& -2n^{2}\binom{2n-k-2}{n+1} + (2n^{2}-1)\binom{2n-k-3}{n} -3\binom{2n-k-4}{n-3} + 3\binom{2n-k-4}{n-1}\\
 &+(2n-5)\binom{2n-k-3}{n-3} - 2(n-2)^{2}\binom{2n-k-3}{n-2}\\
 &-2(n-1)(n-2)\binom{2n-k-3}{n-1} + 2(n+1)(n-1)\binom{2n-k-3}{n+1}\\
 &+2\binom{2n-k-2}{n-3} - 2(n-2)\binom{2n-k-2}{n-2} +2(n-2)(n-1)\binom{2n-k-2}{n-1}.
\end{align*}

Inserting this into $(6)$ we can apply Lemma \ref{lemmaCK2} again to all the sums over $k$. After considerable simplification we arrive at
\[\tilde{b}_{n} = (3n+5)2^{2n-2} - 2(3n-1)\binom{2n-2}{n-1}.\]

This concludes the proof.

\section{The resulting sequence}
\label{s7}
The above formula yields the integer sequence $\{\tilde{b}_{n}\}_{n \geq 2}$, which could not be found in OEIS~\cite{OEIS}. Its first entries are

\begin{center}
 \begin{tabular}{l | l l l l l l l l l}

$n$&2&3&4&5&6&7&8&9&10\\ \hline
$\tilde{b}_{n}$&24&128&648&3160&14984&69536&317264&1427912&6355080.\\ 
\end{tabular}
\end{center}

\section{Acknowledgements}
 I am particularly grateful to Professor Christian Krattenthaler, for his help with simplification of the summation formula which led
 to the short expression for the sequence in Theorem \ref{thmmain}. I also appreciate the advice of Professor Volodymyr Mazorchuk.

\vspace{1cm}

\noindent Department of Mathematics, Uppsala University, Box 480, SE-751 06, Uppsala, Sweden, email: jonathan.nilsson@math.uu.se

\end{document}